\documentclass[journal,twoside,web]{ieeecolor}
\usepackage{generic}
\usepackage{mo_stylefile}

\def\BibTeX{{\rm B\kern-.05em{\sc i\kern-.025em b}\kern-.08em
    T\kern-.1667em\lower.7ex\hbox{E}\kern-.125emX}}
\begin{document}
    \title{Stable Linear System Identification with Prior Knowledge by Riemannian  Sequential Quadratic Optimization}
    \author{Mitsuaki Obara, Kazuhiro Sato, \IEEEmembership{Member, IEEE}, Hiroki Sakamoto, Takayuki Okuno, and Akiko Takeda
    \thanks{Submitted on December 28, 2021 and revised on September 15, 2023. This work was supported in part by the Japan Society for the Promotion of Science KAKENHI under Grant 23K03899, 22KJ0563, 21J20493, 20K14760, 20K19748 and 19H04069, and was part of the results of Value Exchange Engineering, a joint research project between Mercari, Inc. and the RIISE.}
    \thanks{M. Obara is with the Department of Mathematical Informatics, Graduate School of Information Science and Technology, the University of Tokyo, Tokyo, 113-8656 Japan (e-mail: mitsuaki\_obara@mist.i.u-tokyo.ac.jp)}
    \thanks{K. Sato is with the Department of Mathematical Informatics, Graduate School of Information Science and Technology, the University of Tokyo, Tokyo, 113-8656 Japan (e-mail: kazuhiro@mist.i.u-tokyo.ac.jp)}
        \thanks{H. Sakamoto is with the Department of Mathematical Informatics, Graduate School of Information Science and Technology, the University of Tokyo, Tokyo, 113-8656 Japan (e-mail: soccer-books0329@g.ecc.u-tokyo.ac.jp)}
    \thanks{T. Okuno is with the Department of Science and Technology, Faculty of Science and Technology, Seikei University, Tokyo, 180-8633, Japan (e-mail: takayuki-okuno@st.seikei.ac.jp)}
    \thanks{A. Takeda is with the Department of Creative Informatics, Graduate School of Information Science and Technology, the University of Tokyo, Tokyo, 113-8656 Japan, and the Generic Technology Research Group, Center for Advanced Intelligence Project, RIKEN, Tokyo, 103-0037 Japan (e-mail: takeda@mist.i.u-tokyo.ac.jp, akiko.takeda@riken.jp)}
    }
    
    \maketitle
    
    \begin{abstract}
        We consider an identification method for a linear continuous time-invariant autonomous system from noisy state observations.
        In particular, we focus on the identification to satisfy the asymptotic stability of the system with some prior knowledge.
        To this end, we propose to model this identification problem as a Riemannian nonlinear optimization (\RNLO{}) problem, where the stability is ensured through a certain Riemannian manifold and the prior knowledge is expressed as nonlinear constraints defined on this manifold.
        To solve this \RNLO{}, we apply the Riemannian sequential quadratic optimization (\RSQO{}) that was proposed by Obara, Okuno, and Takeda (2022) most recently.  
        {\RSQO{}} performs quite well with theoretical guarantee to find a point satisfying the Karush-Kuhn-Tucker conditions of \RNLO{}. 
        In this paper, we demonstrate that the identification problem can be indeed solved by {\RSQO{}} more effectively than competing algorithms.
    \end{abstract}
    
    
    \section{Introduction} \label{sec:introduction}
        \IEEEPARstart{S}{ystem} identification of a linear continuous time-invariant autonomous system 
        \begin{align}\label{eq:conttimesys}
            \dot{\state[]}\paren*{\timevar} = \A \state[] \paren*{\timevar}
        \end{align}
        with the state vector $\state[]\paren*{\timevar} \in \setR[\dime]$ is a task to estimate $\A\in\setR[\dime\times\dime]$ from measured state data and is one of the most important topics to predict the future state of a real system.
        System identification has been investigated for many years in several settings.
        For example, prediction error methods~\cite{Ljung99SysIdentifBook,Satoetal20RiemIdentifSymPosSys,SatoSato17RiemSysIdentifLinearMIMO,McKelveyetal04DataDrivenLocalCoordforMultvarLinSys,WillsNinness09GradbsdSearchforMultvarSysEsts} and subspace identification methods~\cite{Katayama05SubspaceMethodsBook,Campietal08IterIdentifMethodLinContTimeSys,OverscheeMoor94N4SID,VerhaegenHansson16N2SID,Verhaegen94MOESPIdentifDeterministicPartMIMO} are conventional identification methods for linear time-invariant state space systems.
        Dynamic mode decomposition (\DMD{}) deals with the system of the particular form \cref{eq:conttimesys}, which can be possibly large-scale one such as fluid flows~\cite{Kutzetal16DMDBook, HeilandUnger21IdentifLTIwithDMD}.
        
In real applications, it is often crucial to identify a system satisfying the asymptotic stability.Yet, it is completely non-trivial to ensure the stability in the identification methods such as those above; 
        for example, \DMD{} spectrum has been shown to be considerably sensitive to measurement noise and hence its resultant system may be unstable~\cite{Dawsonetal16CharactCorrSensorNoiseDMD}.
        Several approaches have been investigated to overcome the difficulty: Subspace identification methods with guaranteed stability have been proposed in~\cite{Gesteletal01IdentifStabModelinSubspcIdentifUsingReg,LacyBernstein03SubspcIdentifwithStabUsingCstrOpt, MillerCallafon13SubspcIdentifwithEigenCstrs}.
        In \cite{boots2007constraint}, a constraint generation approach has been proposed.
        A Lagrange relaxation is used to ensure several types of stability including the asymptotic stability for linear time-invariant state space models~\cite{Umenbergeretal18MaxLHIdentifStabLinDynSys} as well as a different stability called the global incremental $\ell^2$ stability for nonlinear systems~\cite{UmenbergerManchester19ConvBoundfoeEqErrinStabNonlinIdentif,UmenbergerManchester19SpecialIPMforstable}. 
        It is worth mentioning that, among the above identification methods, those presented in  \cite{MillerCallafon13SubspcIdentifwithEigenCstrs,LacyBernstein03SubspcIdentifwithStabUsingCstrOpt,Gesteletal01IdentifStabModelinSubspcIdentifUsingReg,Umenbergeretal18MaxLHIdentifStabLinDynSys, boots2007constraint, UmenbergerManchester19ConvBoundfoeEqErrinStabNonlinIdentif,UmenbergerManchester19SpecialIPMforstable} are based on convex optimization.

        In addition to the stability, another important characteristic in the identification is prior knowledge.
        As described in \cite[Chapter 16]{Ljung99SysIdentifBook}, identification should reflect the prior knowledge peculiar to the system, such as the nonnegativity of all or partial components of the system $A$.
        Nevertheless, the conventional identification methods are not able to generate a stable matrix $A$ with prior knowledge information.
        Indeed, to impose the knowledge,
   it is necessary to deal with their nonconvexity, which requires additional techniques from nonconvex optimization.
        
        In this paper, we propose a new method for identifying a system, from noisy data, that satisfies both asymptotic stability and prior knowledge. This method first discretizes the system, and then applies the prediction error method, which is formulated as solving a least squares optimization problem. However, to identify the desired system, we must incorporate both stability and prior knowledge into this optimization problem.     
 Motivated by the fact that an asymptotically stable linear system can be expressed in a port-Hamiltonian form \cite[Proposition 1]{Prajnaetal02LMItoStabLinPortCtrlHamilSys}, we handle the stability of the system through a Riemannian manifold as in \cite{Sato19RiemModelReductofStabLinSys}. Meanwhile, 
 we deal with the prior knowledge information as nonlinear constraints defined on the manifold.
After all, our identification method reduces to solving a constrained nonlinear optimization problem on a Riemannian manifold, called \RNLO{} for short.

        \RNLO{} is one of the general optimization classes proposed in \cite{Yangetal14OptCond} and researches for algorithms to solve \RNLO{} are in progress, for example, \cite{LiuBoumal19Simple,Obaraetal20RiemSQO}.
        Remarkably, the modeling enables us to identify a stable $A$ even when we use noisy observations, which usually spoil the stability in system identification.
        The modeling also allows us to take various prior knowledge into account by virtue of the generality of \RNLO{}.
        To solve \RNLO{}, we propose to apply the Riemannian sequential quadratic optimization method (\RSQO{}) that was presented recently in \cite{Obaraetal20RiemSQO}.
Strength of \RSQO{} is the global convergence property to Karush-Kuhn-Tucker (KKT) points of \RNLO{} and locally fast convergence speed.

        Our contributions are summarized as follows:
        \begin{enumerate}
            \item We propose a prediction error method to identify a linear system satisfying the asymptotic stability with prior knowledge from noisy state observations.
            The key ingredients of this method are formulating the identification problem as {\RNLO}
            and solving it with \RSQO{}.
            \item We conduct numerical experiments with comparisons to demonstrate that \RNLO{} modeling and \RSQO{} are very promising.
        \end{enumerate}

    \subsection*{Organization of paper}
        The rest of this paper is organized as follows.
        In the following subsection, we introduce notations and terminologies.
        In \cref{sec:probsetting}, we formulate the system identification problem as \RNLO{}.
        In \cref{sec:optmethod},  we introduce \RSQO{} of the form tailored to the obtained \RNLO{} and moreover show its theoretical properties.
        We also exploit the geometry of the problem.
        In \cref{sec:experiment}, we demonstrate that the proposed formulation and \RSQO{} are promising through experimental comparisons.
        Finally, in \cref{sec:conclusion}, we conclude this paper with some remarks and discussions about future work.

    \subsection*{Notations and terminologies}
        The sets of real numbers and complex ones are denoted by $\setR$ and $\setC$, respectively.
        Let $\setN$ be the set of all the natural numbers, that is, $\setN := \brc*{1,2,\ldots}$.
        We denote the identity matrix of size $\dime$ by $\eyemat \in \setR[\dime \times \dime]$.
        Given a matrix $\mat \in \setR[\dime \times \dime]$, $\trsp{\mat}$ and $\tr\paren*{\mat}$ represent the transpose of $\mat$ and the trace of $\mat$, respectively.
        Let $\fronorm{\cdot}$ denote the Frobenius norm of a matrix, that is, $\fronorm{\mat} \coloneqq \sqrt{\tr\paren*{\trsp{\mat}\mat}}$ for any $\mat \in \setR[\dime \times \dime]$.
        Let $\norm{\cdot}$ denote the Euclidean norm of a vector, that is, $\norm{\vvec} \coloneqq \sqrt{\vvec[1]^{2} + \cdots + \vvec[\dime]^{2}}$ for any $\vvec \in \setR[\dime]$. 
        Define the $\vecidx$-th standard basis of $\setR[\dime]$ by $\stdbasis \coloneqq \paren*{0, \ldots, 0, 1, 0, \ldots, 0} \in \setR[\dime]$, where the $i$-th element is one and the others are zeros. %
        Given a sufficiently smooth function $\functhr\colon\setR[\dime \times \dime]\to\setR[]$,
        the Euclidean gradient of $\functhr$ at $\mat \in \setR[\dime \times \dime]$ is denoted by $\gradEucli \functhr \paren*{\mat}$.
        Let $\vecspco$ be any vector space and $\tanspc[\Wvec]\vecspco$ be a tangent space at $\Wvec\in\vecspco$.
        We canonically identify $\tanspc[\Wvec]\vecspco$ with $\vecspco$ and use the symbol $\simeq$ to denote the canonical identification.
        Under $\tanspc[\Wvec]\vecspco\simeq\vecspco$,
        given a sufficiently smooth function $\funcfive:\vecspco\to\setR$,
        we define the directional derivative of $\funcfive$ at $\Wvec \in \vecspco$ along $\tanvecone[\Wvec] \in \tanspc[\Wvec]\vecspco$ by
            $\D\funcfive\paren*{\Wvec}\sbra*{\tanvecone[\Wvec]} \coloneqq \lim_{\limparam \downarrow 0} \frac{\funcfive\paren*{\Wvec + \limparam\tanvecone[\Wvec]} - \funcfive\paren*{\Wvec}}{\limparam} \in \setR[]$.

    \section{Preliminaries}
       To formulate our problem as \RNLO{},
       we characterize the stability of the system in \cref{subsec:charstab}.
        To this end, we first summarize the Riemannian geometries of the sets of skew-symmetric matrices and  
        symmetric positive definite matrices.

       \subsection{Riemannian geometries of skew-symmetric matrices and symmetric positive definite matrices}

        We define the vector space of the skew-symmetric matrices in $\setR[\dime \times \dime]$ by
            $\skewmani \coloneqq \brc*{\J \in \setR[\dime \times \dime] \relmiddle{|} \trsp{\J} = - \J}$.
        $\skewmani$ can be regarded as a linear manifold, whose tangent space at $\J \in \skewmani$, denoted by $\tanspc[\J]\skewmani$, is canonically identified with $\skewmani$ itself.
        Under $\tanspc[\J]\skewmani \simeq \skewmani$, the Riemannian metric at $\J$ is defined as
            $\metr[\J]{\tanvecone[\J]}{\tanvectwo[\J]} \coloneqq \tr\paren*{\trsp{\tanvecone[\J]} \tanvectwo[\J]}$
        for any $\tanvecone[\J], \tanvectwo[\J] \in \skewmani$.
        Define an operator $\skewstr\colon\setR[\dime\times\dime]\to\skewmani$ by
            $\skewstr\paren*{\mat[]} \coloneqq \frac{\mat[] - \trsp{\mat[]}}{2}$.
        Let $\funcone\colon\skewmani\to\setR[]$ be a twice continuously differentiable function and $\extfuncone$ be the smooth extension of $\funcone$ to the Euclidean space $\setR[\dime\times\dime]$.
        The Riemannian gradient of $\funcone$ at $\J$ is
        \begin{align}\label{eq:Riemgradskew}
            \gradstr\funcone\paren*{\J} = \skewstr\paren*{\gradEucli\extfuncone\paren*{\J}} \in \skewmani,
        \end{align}
        where $\gradEucli\extfuncone\paren*{\J}$ is the Euclidean gradient of $\extfuncone$ at $\J$.
        The exponential mapping at $\J$ is given by
            $\expmap[\J]\paren*{\tanvecone[\J]} = \J + \tanvecone[\J]$. 
        for all $\tanvecone[\J] \in \skewmani$.
        
        Let us denote the set of symmetric positive definite matrices in $\setR[\dime \times \dime]$ by $\symposmani$, which is an open submanifold of $\setR[\dime \times \dime]$.
        The tangent space at $\Pmat \in \symposmani$, denoted by $\tanspc[\Pmat]\symposmani$, is canonically identified with the set of symmetric matrices denoted by $\symset$.
        Under $\tanspc[\Pmat]\symposmani\simeq\symset$, we equip $\symposmani$ with a Riemannian metric at $\Pmat \in \symposmani$ defined by
        \begin{align}\label{eq:Riemmetrsympos}
            \metr[\Pmat]{\tanvecone[\Pmat]}{\tanvectwo[\Pmat]} \coloneqq \tr\paren*{\inv{\Pmat}\tanvecone[\Pmat]\inv{\Pmat}\tanvectwo[\Pmat]}
        \end{align}
        for any $\tanvecone[\Pmat], \tanvectwo[\Pmat] \in \symset$~\cite{Jeuris12SurveyComparisonMatGeoMean}.
        Let us define an operator $\sym\colon\setR[\dime\times\dime]\to\symset$ by
            $\sym\paren*{\mat[]} \coloneqq \frac{\mat[]+ \trsp{\mat[]}}{2}$.
        Let $\functwo\colon\symposmani\to\setR$ be a twice continuously differentiable function and $\extfunctwo$ be the smooth extension of $\functwo$ to the Euclidean space $\setR[\dime\times\dime]$.
        The Riemannian gradient of $\functwo$ at $\Pmat$ is
        \begin{align}\label{eq:Riemgradsympos}
            \gradstr\functwo\paren*{\Pmat} = \Pmat\sym\paren*{\gradEucli\extfunctwo\paren*{\Pmat}}\Pmat,
        \end{align}
        where $\gradEucli\extfunctwo\paren*{\Pmat}$ is the Euclidean gradient of $\extfunctwo$ at $\Pmat$.
        On $\symposmani$ equipped with \cref{eq:Riemmetrsympos}, we introduce retraction \cite{Jeuris12SurveyComparisonMatGeoMean} by %
            $\retr[\Pmat]\paren*{\tanvecone[\Pmat]} = \Pmat + \tanvecone[\Pmat] + \frac{1}{2}\tanvecone[\Pmat]\inv{\Pmat}\tanvecone[\Pmat]$
        for all $\tanvecone[\Pmat] \in \symset$.
        See \cite{Absiletal08OptBook, Jeuris12SurveyComparisonMatGeoMean}
        for the details of the concepts and the notations on optimization on Riemannian manifolds and the geometries of the $\skewmani$ and $\symposmani$.

        \subsection{Characterizations of stability}\label{subsec:charstab}
        Let us start with the definitions of the stability: $A \in \setR[\dime \times \dime]$ is stable if the real parts of all the eigenvalues of the matrix $\A$ are negative.
        We say that the system \cref{eq:conttimesys} is asymptotically stable if $\A$ is stable.
        
        In the following proposition, a useful result for the stability is provided.
        The proof can be found in the literature, e.g., \cite[{Section III}]{Sato19RiemModelReductofStabLinSys} and \cite{Prajnaetal02LMItoStabLinPortCtrlHamilSys}.
        \begin{prop}\label{prop:charactstab}
            A matrix $\A \in \setR[\dime \times \dime]$ is stable if and only if there exists $\paren*{\J,\R,\Q} \in \skewmani \times \symposmani \times \symposmani$ such that 
            \begin{align}\label{eq:AJRQ}
                \A = \paren*{\J - \R} \Q.
            \end{align}
        \end{prop}
        This proposition motivates us to consider optimization problems with respect to $\paren*{\J,\R,\Q}$ to ensure the stability of the system.
        Note that, for any stable matrix $\A$, the triplet $\paren*{\J, \R, \Q}$ satisfying \cref{eq:AJRQ} is not unique.
        
        Define 
            $\mani \coloneqq \skewmani \times \symposmani \times \symposmani$
        and write
            $\prodvar \coloneqq \paren*{\J,\R,\Q} \in \mani$.
        Clearly, $\mani$ is a product Riemannian manifold and its tangent space at $\prodvar$ is expressed as
            $\tanspc[\prodvar]\mani \simeq \skewmani \times \symset \times \symset$,
        on which a Riemannian metric at $\prodvar$ is defined as
            $\metr[\prodvar]{\tanvecone[\prodvar]}{\tanvectwo[\prodvar]} \coloneqq \tr\paren*{\trsp{\tanvecone[\J]}\tanvectwo[\J]} + \tr\paren*{\inv{\R}\tanvecone[\R]\inv{\R}\tanvectwo[\R]} + \tr\paren*{\inv{\Q}\tanvecone[\Q]\inv{\Q}\tanvectwo[\Q]}$
        for any $\tanvecone[\prodvar] = \paren*{\tanvecone[\J], \tanvecone[\R], \tanvecone[\Q]}, \tanvectwo[\prodvar] = \paren*{\tanvectwo[\J], \tanvectwo[\R], \tanvectwo[\Q]} \in \tanspc[\prodvar]\mani$.
        We use the retraction
            $\retr[\prodvar]\paren*{\tanvecone[\prodvar]}
            = \paren*{\J+\tanvecone[\J], \R + \tanvecone[\R] + \frac{1}{2}\tanvecone[\R]\inv{\R}\tanvecone[\R], \Q + \tanvecone[\Q] + \frac{1}{2}\tanvecone[\Q]\inv{\Q}\tanvecone[\Q]}$
        for all $\tanvecone[\prodvar] = \paren*{\tanvecone[\J], \tanvecone[\R], \tanvecone[\Q]}\in \tanspc[\prodvar]\mani$.
        We refer readers to \cite{Absiletal08OptBook} for the geometry of a product manifold.

        \section{Problem setup} \label{sec:probsetting}
        In this section, we formulate optimization problems for the identification of a stable linear system with prior knowledge and derive optimality conditions.
        
        \subsection{Problem formulations}\label{subsec:prederr}
        In this paper, we assume that a system to be identified is of the form \cref{eq:conttimesys} with stable $A$. Thus, from Proposition \ref{prop:charactstab},
        a discretized model with noise $\varepsilon_k$ of the system by the Euler method is expressed by 
        \begin{align}\label{eq:Eulermethod}
            \state[\stateidx + 1] = \paren*{\eyemat + \sampintvl (J-R)Q} \state + \varepsilon_k,
        \end{align}
        where $\sampintvl > 0$ is the sampling interval.
        That is, the noise $\varepsilon_k$ can be interpreted as the prediction error at time $k$.
        From \cref{eq:Eulermethod}, we can define the least-square function
         \begin{align}
            \objfun\paren*{\prodvar} &\coloneqq \frac{1}{N}\sum_{k=0}^{N-1} \|\varepsilon_k\|^2 \nonumber\\
            &=\frac{1}{\obsnum}\fronorm{\statesetshifted - \paren*{\eyemat + \sampintvl\paren*{\J-\R}\Q}\stateset}^{2}, \label{eq:Riemobjfun}
        \end{align}
        where
            $\stateset \coloneqq 
            \begin{pmatrix}
            \state[0] & \state[1]& \cdots  &\state[\obsnum-1]
            \end{pmatrix}\in \setR[\dime \times \obsnum]$ and
            $\statesetshifted \coloneqq 
            \begin{pmatrix}
            \state[1] & \state[2] & \cdots & \state[\obsnum]
            \end{pmatrix}\in \setR[\dime \times \obsnum]$.
Note that the system \eqref{eq:Eulermethod} is
a nonlinear regression model,
because there exists a multiplication of the parameters $J-R$ and $Q$.

        Using the least-square function \eqref{eq:Riemobjfun},
        we consider the following optimization problem:
        \begin{mini!}
            {\prodvar=\paren*{\J,\R,\Q} \in \mani}{\objfun\paren*{\prodvar}\label{eq:objfunRiemsysidentif}}%
            {\label{eq:Riemsysidentif}}{}
            \addConstraint{\ineqfun\paren*{\prodvar}}{\leq 0,}{\quad\paren*{\paren*{\matrowidx,\matcolidx} \in \ineqset}}
            \addConstraint{\eqfun\paren*{\prodvar}}{=0,}{\quad\paren*{\paren*{\matrowidx,\matcolidx} \in \eqset}.}
        \end{mini!}
        Here, $\ineqset, \eqset \subseteq \brc*{1,\ldots,\dime} \times \brc*{1,\ldots,\dime}$ are the index sets of inequality and equality constraints, respectively.
        Moreover, $\brc*{\ineqfun}_{\paren*{\matrowidx,\matcolidx} \in \ineqset}, \brc*{\eqfun}_{\paren*{\matrowidx,\matcolidx} \in \eqset}\colon \mani\to\setR[]$ are continuously differentiable functions.
        The problem~\cref{eq:Riemsysidentif} is a Riemannian nonlinear optimization problem (\RNLO{})~\cite{Yangetal14OptCond}, that is, a constrained nonlinear optimization problem over a Riemannian manifold.
         It should be noted that this formulation is novel for a stable linear system identification with prior knowledge.

        Any inequality and equality constraints are acceptable in \cref{eq:Riemsysidentif} as long as they are continuously differentiable with respect to $\prodvar$; for example, we can deal with nonnegativity (or nonpositivity) of the elements
            $\trsp{\stdbasis[\matrowidx]} \paren*{\J-\R}\Q \stdbasis[\matcolidx] \lesseqgtr 0$
        and an element equality, 
            $\trsp{\stdbasis[\matrowidx]} \paren*{\J-\R}\Q \stdbasis[\matcolidx] = \const[\matrowidx\matcolidx]$
        for a given constant $\const[\matrowidx\matcolidx] \in \setR[]$.
        In \cref{sec:experiment}, we consider box-type inequalities.

        \subsection{Optimality conditions}\label{subsec:optcond}
        In what follows, we define the Karush-Kuhn-Tucker (KKT) conditions and relevant concepts for \RNLO{}.
        We say that $\prodvaropt \in \mani$ satisfies the KKT conditions of \RNLO{}~\cref{eq:Riemsysidentif} if there exist Lagrange multipliers $\brc*{\ineqLagmultopt}_{\paren*{\matrowidx, \matcolidx} \in \ineqset} \subseteq \setR[]$ and $\brc*{\eqLagmultopt}_{\paren*{\matrowidx, \matcolidx} \in \eqset} \subseteq \setR[]$ such that the following hold:
	    \begin{align}\label{eq:KKTconds}
	        \begin{aligned}
	            &\gradstr\objfun\paren*{\prodvaropt}\\
	            &\qquad +\sum_{\paren*{\matrowidx, \matcolidx} \in \ineqset} \ineqLagmultopt \gradstr\ineqfun\paren*{\prodvar} + \sum_{\paren*{\matrowidx, \matcolidx} \in \eqset} \eqLagmultopt \gradstr\eqfun\paren*{\prodvar} = 0,\\
    	        &\ineqLagmultopt \geq 0, \ineqfun\paren*{\prodvaropt} \leq 0, \ineqLagmultopt \ineqfun\paren*{\prodvaropt} = 0, \quad \paren*{\paren*{\matrowidx, \matcolidx} \in \ineqset},\\
    	        &\eqfun\paren*{\prodvaropt} = 0, \quad \paren*{\paren*{\matrowidx, \matcolidx} \in \eqset},
            \end{aligned}
	    \end{align}
	    where the operator $\gradstr$ denotes the Riemannian gradient as in \cref{sec:introduction}.
        We call $\prodvaropt$ a KKT point of \cref{eq:Riemsysidentif}.
        For brevity, we often write $\ineqLagmultopt[] \in \setR[\abs{\ineqset}]$ and $\eqLagmultopt[] \in \setR[\abs{\eqset}]$ for $\brc*{\ineqLagmultopt}_{\paren*{\matrowidx, \matcolidx} \in \ineqset}$ and $\brc*{\eqLagmultopt}_{\paren*{\matrowidx, \matcolidx} \in \eqset}$, respectively.
        It is known that, under some conditions called constraint qualifications, the KKT conditions are necessary ones for the optimality; that is, a local minimizer satisfies the KKT conditions under constraint qualifications~\cite{Yangetal14OptCond,BergmannHerzog19IntKKT}.

    \section{Optimization methods}\label{sec:optmethod}
                 In this section, we introduce a specific method for solving \cref{eq:Riemsysidentif}, Riemannian sequential quadratic optimization (\RSQO{}) that was originally presented in \cite{Obaraetal20RiemSQO}.
        
        \subsection{Geometry of problems}\label{subsec:geometryofprobs}
            In the following theorem, we derive the Riemannian gradient of $\objfun$ specifically.
            \begin{theo}\label{prop:RiemgradHessofobjfuncs}
                Given $\prodvar=\paren*{\J,\R,\Q} \in \mani$, the Riemannian gradient of $\objfun$ in \cref{eq:Riemsysidentif} is
                \begin{align}
                    \begin{split}\label{eq:Riemgradobjfun}
                        &\gradstr\objfun\paren*{\prodvar}\\ &=(\skewstr\paren*{\gradpartobjfun[\J]\paren*{\prodvar}},\R\sym\paren*{\gradpartobjfun[\R]\paren*{\prodvar}}\R, \Q\sym\paren*{\gradpartobjfun[\Q]\paren*{\prodvar}}\Q),
                    \end{split}
                \end{align}
                where $\extobjfun$ is the smooth extension of $\objfun$ to the Euclidean space and
                \begin{align}
                    &\gradpartobjfun[\J]\paren*{\prodvar} = - \frac{2 \sampintvl}{\obsnum} \paren*{\statesetshifted - \paren*{\eyemat + \sampintvl \paren*{\J - \R}\Q} \stateset}\trsp{\stateset}\trsp{\Q}, \label{eq:EucligradobjfuncJ}\\
                    &\gradpartobjfun[\R]\paren*{\prodvar} = - \gradpartobjfun[\J]\paren*{\prodvar}, \label{eq:EucligradobjfuncR}\\  %
                    &\gradpartobjfun[\Q]\paren*{\prodvar} = -\frac{2 \sampintvl}{\obsnum} \paren*{\trsp{\J} - \trsp{\R}}\paren*{\statesetshifted - \paren*{\eyemat + \sampintvl \paren*{\J - \R}\Q}\stateset}\trsp{\stateset} \label{eq:EucligradobjfuncQ}
                \end{align}
                are the Euclidean gradients on $\J, \R,$ and $\Q$, respectively.
            \end{theo}
            \begin{proof}
 Note that the Riemannian gradient of the function on the product manifold is organized component-wisely~\cite{Absiletal08OptBook}.%
    Using this fact, %
    we derive the Riemannian gradient of our problems.
    Hereafter, we abbreviate $\setR[\dime\times\dime] \times\setR[\dime\times\dime]\times\setR[\dime\times\dime]$ as $\thrprodsetR$ for brevity.
            
              Let $\extobjfun$ be the smooth extension of $\objfun$ to the Euclidean space.
        We define $\extobjEulerfun\colon \thrprodsetR\to\setR[\dime]$ by
        \begin{align}\label{eq:extobjEulerfun}
            \extobjEulerfun\paren*{\extprodvar} \coloneqq \state[\stateidx + 1] - Z \state,
        \end{align}
        where $\extprodvar \coloneqq \paren*{\extJ, \extR, \extQ} \in \thrprodsetR$ and $Z \coloneqq \eyemat + \sampintvl \paren*{\extJ - \extR}\extQ$
        By substituting \cref{eq:extobjEulerfun} into \cref{eq:Riemobjfun}, we have
            $\extobjfun\paren*{\extprodvar} = \frac{1}{N} \sum_{\stateidx = 1}^{\obsnum -1} \trsp{\extobjEulerfun\paren*{\extprodvar}}\extobjEulerfun\paren*{\extprodvar}$
        and its directional derivative at $\extprodvar \in \thrprodsetR$ along $\exttanvecone[\extprodvar] = \paren*{\exttanvecone[{\extJ}], \exttanvecone[\extR], \exttanvecone[\extQ] } \in \thrprodsetR$ is 
        \begin{align}\label{eq:dirderivextobjfunwithextobjEulerfun}
            \D\extobjfun\paren*{\extprodvar}\sbra*{\exttanvecone[\extprodvar]} = \frac{2}{\obsnum} \trsp{\paren*{\D\extobjEulerfun\paren*{\extprodvar}\sbra*{\exttanvecone[\extprodvar]}}}\extobjEulerfun\paren*{\extprodvar}.
        \end{align}
        
        Now, we derive the explicit form of $\D\extobjEulerfun\paren*{\extprodvar}\sbra*{\exttanvecone[\extprodvar]}$ in \cref{eq:dirderivextobjfunwithextobjEulerfun}.
        To this end, we first calculate the directional derivative along $\paren*{\exttanvecone[{\extJ}], 0, 0}$ as
                $\D\extobjEulerfun\paren*{\extprodvar}\sbra*{\paren*{\exttanvecone[{\extJ}], 0, 0}}
                = - \sampintvl \exttanvecone[{\extJ}] \extQ \state$.
        Similarly, we have
            $\D\extobjEulerfun\paren*{\extprodvar}\sbra*{\paren*{0,\exttanvecone[{\extR}],0}} = \sampintvl \exttanvecone[{\extR}] \extQ \state$
          and $\D\extobjEulerfun\paren*{\extprodvar}\sbra*{\paren*{0, 0, \exttanvecone[{\extQ}]}} = - \sampintvl \paren*{\extJ - \extR} \exttanvecone[{\extQ}] \state$. %
        Thus, the explicit form of $\D\extobjEulerfun\paren*{\extprodvar}\sbra*{\exttanvecone[\extprodvar]}$ is
        \begin{align}
            \begin{split}\label{eq:explicitdirderivextobjfunwithextobjEulerfun}
               \D\extobjEulerfun\paren*{\extprodvar}\sbra*{\exttanvecone[\extprodvar]}
                = - \sampintvl \paren*{ \paren*{\exttanvecone[\extJ] - \exttanvecone[\extR]} \extQ + \paren{\extJ - \extR} \exttanvecone[\extQ]} \state.
            \end{split}
        \end{align}
        
        Using the result, we derive the Riemannian gradient of $\objfun$.
        By substituting \cref{eq:extobjEulerfun} and \cref{eq:explicitdirderivextobjfunwithextobjEulerfun} into \cref{eq:dirderivextobjfunwithextobjEulerfun}, we have
        \begin{align}
                & \D\extobjfun\paren*{\extprodvar}\sbra*{\exttanvecone[\extprodvar]}
                = \tr\paren*{\trsp{\exttanvecone[\extJ]} \sum_{\stateidx = 1}^{\obsnum - 1} - \frac{2 \sampintvl}{\obsnum}\bigl(\state[\stateidx + 1]- Z \state\bigr) \trsp{\state} \trsp{\extQ}} \nonumber\\
                &\quad+ \tr\Biggl(\trsp{\exttanvecone[\extR]} \sum_{\stateidx = 1}^{\obsnum - 1} \frac{2 \sampintvl}{\obsnum}\bigl(\state[\stateidx + 1] - Z \state\bigr) \trsp{\state} \trsp{\extQ}\Biggr) \label{eq:dirderivextobjfun}\\ 
                &\quad + \tr\paren*{\trsp{\exttanvecone[\extQ]} \sum_{\stateidx = 1}^{\obsnum - 1} -\frac{2 \sampintvl}{\obsnum}\paren*{\trsp{\extJ} - \trsp{\extR}}
                 \paren*{\state[\stateidx + 1] - Z \state} \trsp{\state}}.\nonumber
        \end{align}
        Here, it follows from $\state[\stateidx + 1] = \statesetshifted\stdbasis[\stateidx]$ and $\state[\stateidx] = \stateset\stdbasis[\stateidx]$ that
        \begin{align}
                \sum_{\stateidx = 1}^{\obsnum - 1} \paren*{\state[\stateidx + 1] - Z \state} \trsp{\state}
                &= \paren*{\statesetshifted - Z \stateset} \paren*{ \sum_{\stateidx = 1}^{\obsnum - 1} \stdbasis[\stateidx] \trsp{\stdbasis[\stateidx]}} \trsp{\stateset} \nonumber\\
                &= \paren*{\statesetshifted - Z \stateset}\trsp{\stateset},\label{eq:partdirderivextobjfun}
        \end{align}
        where the second equality holds from $\sum_{\stateidx = 1}^{\obsnum - 1} \stdbasis[\stateidx] \trsp{\stdbasis[\stateidx]} = \eyemat$.
        Thus, since $\mani\subseteq\thrprodsetR$ holds, combining \cref{eq:dirderivextobjfun} and \cref{eq:partdirderivextobjfun} %
        provides \eqref{eq:EucligradobjfuncJ}, \eqref{eq:EucligradobjfuncR}, and \eqref{eq:EucligradobjfuncQ}
        for any $\prodvar\in\mani$.
        By projecting \cref{eq:EucligradobjfuncJ}, \cref{eq:EucligradobjfuncR}, and \cref{eq:EucligradobjfuncQ} onto $\tanspc[\prodvar]\mani$ according to \cref{eq:Riemgradskew} and \cref{eq:Riemgradsympos},%
        we obtain the Riemannian gradient of the form \cref{eq:Riemgradobjfun}.
            \end{proof}    
            In a similar manner, we can derive the Riemannian gradient of constraints although we focus on that of the objective function here.
        
        \subsection{Riemannian sequential quadratic optimization}\label{subsec:eRSQO}
            In this subsection, we briefly explain the \RSQO{} for solving \cref{eq:Riemsysidentif} using notations and terminologies we have set up so far. See \cite[Section~3]{Obaraetal20RiemSQO} for more details.

  \subsubsection{Description of \RSQO{}}\label{subsec:descrofprpsdalgo}
                \RSQO{} is an iterative method. %
                Let $\prodvar[\iteridx] \in \mani$ be a current iterate. 
                \RSQO{} first solves the following quadratic subproblem to generate a search direction $\solvarrsqosub\in\tanspc[{\prodvar[\iteridx]}]\mani$:
                \begin{mini}
                    {\varrsqosub \in \tanspc[{\prodvar[\iteridx]}]\mani}{\quadobjfunrsqosub\paren*{\varrsqosub}}
                    {\label{eq:RSQOsubprob}}{}
                    \addConstraint{\linineqfunrsqosub\paren*{\varrsqosub}}{\leq 0,}{\quad \paren*{\paren*{\matrowidx,\matcolidx} \in \ineqset}}
                    \addConstraint{\lineqfunrsqosub\paren*{\varrsqosub}}{= 0,}{\quad \paren*{\paren*{\matrowidx,\matcolidx} \in \eqset}},
                \end{mini}
                where 
                \begin{align}
                    &\quadobjfunrsqosub\paren*{\varrsqosub} \coloneqq \frac{1}{2}\metr[{\prodvar[\iteridx]}]{\oprrsqosub\sbra{\varrsqosub}}{\varrsqosub} + \metr[{\prodvar[\iteridx]}]{\gradstr\objfun\paren*{\prodvar[\iteridx]}}{\varrsqosub},\label{def:objfunRSQOsubprob}\\
                    &\linineqfunrsqosub\paren*{\varrsqosub}\coloneqq \ineqfun\paren*{\prodvar[\iteridx]} + \metr[{\prodvar[\iteridx]}]{\gradstr\ineqfun\paren*{\prodvar[\iteridx]}}{\varrsqosub},\label{def:linineqRSQOsubprob}\\
                    &\lineqfunrsqosub\paren*{\varrsqosub}\coloneqq \eqfun\paren*{\prodvar[\iteridx]} + \metr[{\prodvar[\iteridx]}]{\gradstr\eqfun\paren*{\prodvar[\iteridx]}}{\varrsqosub},\label{def:lineqRSQOsubprob}
                \end{align}
                and $\oprrsqosub\colon\tanspc[{\prodvar[\iteridx]}]\mani\to\tanspc[{\prodvar[\iteridx]}]\mani$ is a symmetric positive-definite linear operator.
                Since the problem can be expressed as a Euclidean convex quadratic optimization problem, it can be solved by, for example, interior-point methods or active set methods.
 We denote the Lagrange multipliers 
 at the solution $\solvarrsqosub$ 
 corresponding to the inequality and equality constraints
 by $\solineqLagmultrsqosub[] \in \setR[\abs{\ineqset}]$ and $\soleqLagmultrsqosub[] \in \setR[\abs{\eqset}]$, respectively.

                Next, \RSQO{} determines the step length by using the $\ell_{1}$ penalty function defined as
                \begin{align*}
                    \ellonepenafun[]\paren*{\prodvar} \coloneqq \objfun\paren*{\prodvar} + \penaparamrsqo[]\paren*{\sum_{\paren*{\matrowidx,\matcolidx} \in \ineqset} \max\brc*{0,\ineqfun\paren*{\prodvar}} + \sum_{\paren*{\matrowidx,\matcolidx} \in \eqset} \abs{\eqfun\paren*{\prodvar}}},
                \end{align*}
                where $\penaparamrsqo[] > 0$ is a penalty parameter~\cite{LiuBoumal19Simple}.
                \RSQO{} first sets
                \begin{align}\label{eq:updatepenaparam}
                    \penaparamrsqo =
                    \begin{cases}
                        \penaparamrsqo[\iteridx - 1], & \text{if } \penaparamrsqo[\iteridx - 1] \geq \maxabsLagmult, \\
                        \maxabsLagmult + \svalpenaupd, & \text{otherwise},
                    \end{cases}
                \end{align}
                where $\maxabsLagmult \coloneqq \max\brc*{\max_{\paren*{\matrowidx, \matcolidx} \in \ineqset} \solineqLagmultrsqosub, \max_{\paren*{\matrowidx,\matcolidx} \in \eqset} \abs{\soleqLagmultrsqosub}}$ and $\svalpenaupd > 0$ is a prescribed algorithmic parameter.
                Then, by using $\ellonepenafun$ as a merit function, we find the smallest nonnegative integer $\armijoint$ satisfying
                \begin{align}
                    \begin{split}\label{eq:linesearchrule}
                        &\armijoconst \armijoratio[\armijoint] \metr[{\prodvar[\iteridx]}]{\oprrsqosub\sbra{\solvarrsqosub}}{\solvarrsqosub} \\
                        &\quad \leq \ellonepenafun\paren*{\prodvar[\iteridx]} - \ellonepenafun \paren*{\retr[{\prodvar[\iteridx]}] \paren*{\armijoratio[\armijoint] \solvarrsqosub}}
                    \end{split}
                \end{align}
                and set $\steplength = \armijoratio[\armijoint]$.
                Using the search direction and the step length, \RSQO{} updates the iterate.
                The above procedure is formalized as in \cref{algo:eRSQO}.
                
                \begin{algorithm}[t]
                    \SetKwInOut{Require}{Require}
                    \SetKwInOut{Input}{Input}
                    \SetKwInOut{Output}{Output}
                    \Input{Initial point $\prodvar[0] \in \mani$; initial linear operator $\oprrsqosub[0]\colon\tanspc[{\prodvar[0]}]\mani\to\tanspc[{\prodvar[0]}]\mani$; hyperparameters
                    $\penaparamrsqo[-1] > 0, \svalpenaupd > 0, \armijoratio[] \in \paren*{0,1}, \armijoconst \in \paren*{0,1}$; } %
                    \For{$\iteridx = 0,1,\ldots$}{
                        Compute $\solvarrsqosub$ with associated Lagrange multipliers $\solineqLagmultrsqosub$ and $\soleqLagmultrsqosub$ by solving \cref{eq:RSQOsubprob}\;
                        Update $\penaparamrsqo$ according to \cref{eq:updatepenaparam}\;
                        Determine the integer $\armijoint$ according to the line search rule \cref{eq:linesearchrule} and set $\steplength = \armijoratio[\armijoint]$\;
                        Update $\prodvar[\iteridx + 1] = \retr[{\prodvar[\iteridx]}]\paren*{\steplength\solvarrsqosub},
                            \ineqLagmult^{\iteridx + 1} = \solineqLagmultrsqosub$ for $\paren*{\matrowidx,\matcolidx} \in \ineqset$, and $\eqLagmult^{\iteridx + 1} = \soleqLagmultrsqosub$ for $\paren*{\matrowidx,\matcolidx} \in \eqset$.
                    }
                \caption{Riemannian sequential quadratic optimization (\RSQO{})}\label{algo:eRSQO}
            \end{algorithm}

\subsubsection{Convergence properties}

First, the line search terminates within finitely many trials in \RSQO{}, because 
the search direction is a descent direction for $\ellonepenafun\circ\retr[{\prodvar[\iteridx]}]$. For details, see \cite[Proposition~3.9, Remark~3.10]{Obaraetal20RiemSQO}.  
Next, we show that \RSQO{} has the global convergence property under the following assumptions, standard in nonlinear optimization. 
            \begin{assu}\label{assu:opruniformbound}
                There exist $\lboprrsqosub,\uboprrsqosub>0$ such that, for any $\iteridx$,
                \begin{align*}
                   \lboprrsqosub \norm{\varrsqosub}_{\prodvar[\iteridx]}^{2} \leq \metr[{\prodvar[\iteridx]}]{\oprrsqosub \sbra*{\varrsqosub}}{\varrsqosub} \leq M \norm{\varrsqosub}_{\prodvar[\iteridx]}^{2}
                \end{align*}
               for all $\varrsqosub \in \tanspc[{\prodvar[\iteridx]}]\mani$.
            \end{assu}
            \begin{assu}\label{assu:feas}
 The subproblem\,\cref{eq:RSQOsubprob} is feasible at each iteration. 
            \end{assu}
            \begin{assu}\label{assu:geneseqbound}
                The generated sequence $\brc*{\paren*{\prodvar[\iteridx], {\ineqLagmult[]}^{\iteridx},{\eqLagmult[]}^{\iteridx}}}_{\iteridx}$ is bounded. Here, the boundedness of
                $\{\prodvar[\iteridx]\}$ is with respect to the distance induced from the Riemannian metric, while 
                those of $\{{\ineqLagmult[]}^{\iteridx}\}$ and $\{{\eqLagmult[]}^{\iteridx}_{\iteridx}\}$ are in the sense of the Euclidean distance.
            \end{assu}
\cref{assu:geneseqbound,assu:feas} are not easy to check in prior. But, we observed that they were not violated in the numerical experiments. By \cref{assu:geneseqbound},
any accumulation point of $\{\theta^k\}$ stays at $\symposmani$.   
 \begin{theo}{\cite[Theorem~3.12]{Obaraetal20RiemSQO}}\label{theo:globconvRSQO}
                Suppose \cref{assu:opruniformbound,assu:feas,assu:geneseqbound}.
                Let $\paren*{\prodvaropt, \ineqLagmult[]^{\optsymbol}, \eqLagmult[]^{\optsymbol}}$ be any accumulation point of $\brc*{\paren*{\prodvar[\iteridx], \ineqLagmult[]^{\iteridx+1}, \eqLagmult[]^{\iteridx+1}}}$ generated by \RSQO{}. 
                Then, $\paren*{\prodvaropt, \ineqLagmult[]^{\optsymbol}, \eqLagmult[]^{\optsymbol}}$ satisfies the KKT conditions of \RNLO{} \cref{eq:Riemsysidentif}.
            \end{theo}

    \section{Numerical simulations}\label{sec:experiment}
    In this section, we demonstrate the effectiveness of our \RNLO{} modeling and \RSQO{}.
    We conduct numerical experiments on a random system and make comparisons.
    In \cref{subsec:probsetrandom}, we introduce a problem setting as well as other modelings for the comparisons.
    In \cref{subsec:exptenv}, we describe the solver settings and an evaluation index for the experiments.
    In \cref{subsec:exptresult}, we show the numerical results and discuss the effect of choice of the modelings and the algorithms.
    All the experiments are implemented in Matlab\_R2023a and Manopt~\cite{Boumaletal14Manopt} on a Windows 10 Pro with 2.60 GHz Core i9-11980HK CPU and 64.0 GB memory.
    
    \subsection{Problem setting: synthetic system}\label{subsec:probsetrandom}
        In this experiment, we consider the following synthetic system:
        \begin{mini!}
            {\prodvar=\paren*{\J,\R,\Q} \in \mani}{\objfun\paren*{\prodvar}=\frac{1}{\obsnum}\fronorm{\statesetshifted - \paren*{\eyemat + \sampintvl\paren*{\J-\R}\Q}\stateset}^{2}\label{eq:Cost1b2bRiemsysidentif}} 
            {\label{eq:1b2bRiemsysidentif}}{}
            \addConstraint{\leftconst \leq \trsp{\stdbasis[\matrowidx]} \paren*{\J-\R}\Q \stdbasis[\matcolidx] \leq \rightconst\notag}
            \addConstraint{\qquad\qquad\qquad\qquad\paren*{\paren*{\matrowidx,\matcolidx} \in \ineqoneboxset\cup\ineqtwoboxset}\label{eq:IneqOne1b2bRiemsysidentif}}
            \addConstraint{\lengthconst^2 \leq \paren*{\trsp{\stdbasis[\matrowidx]} \paren*{\J-\R}\Q \stdbasis[\matcolidx] - \centerconst}^2\notag}
            \addConstraint{\qquad\qquad\qquad\qquad\paren*{\paren*{\matrowidx,\matcolidx} \in \ineqtwoboxset},\label{eq:IneqTwo1b2bRiemsysidentif}}
        \end{mini!}
         where $\ineqoneboxset, \ineqtwoboxset$ are disjoint subsets of the whole indices $\brc*{1,\ldots,\dime}\times\brc*{1,\ldots,\dime}$.
        $\ineqoneboxset$ is the index set of $1$-box constraints, for each $\paren*{\matrowidx,\matcolidx}$ of which the $\paren*{\matrowidx,\matcolidx}$-th component of the true system belongs to $\sbra{\leftconst,\rightconst} \subseteq \setR[]$.
        Similarly, $\ineqtwoboxset$ is the one of $2$-box constraints, for each 
        $\paren*{\matrowidx,\matcolidx}$ of which the true $\paren*{\matrowidx,\matcolidx}$-th element is assumed to lie in $\sbra{\leftconst,\centerconst-\lengthconst}\cup\sbra{\centerconst+\lengthconst,\rightconst}$.
        In a similar manner to \cref{prop:RiemgradHessofobjfuncs}, we derive the Riemannian gradients of the constraints.

        We will compare the following two modelings with the above \RNLO{} modeling.
        One is the Euclidean version of the prediction error method with prior knowledge (Euclidean nonlinear optimization; \ENLO{}):
        \begin{mini!}
            {\A \in \setR[\dime \times \dime]}{\frac{1}{\obsnum}\fronorm{\statesetshifted - \paren*{\eyemat + \sampintvl\A}\stateset}^{2}\label{eq:CostEuclisysidentif}} 
            {\label{eq:Euclisysidentif}}{}
            \addConstraint{\leftconst \leq \trsp{\stdbasis[\matrowidx]} \A \stdbasis[\matcolidx] \leq \rightconst,\paren*{\paren*{\matrowidx,\matcolidx} \in \ineqoneboxset\cup\ineqtwoboxset}}
            \addConstraint{\lengthconst^2 \leq \paren*{\trsp{\stdbasis[\matrowidx]} \A \stdbasis[\matcolidx] - \centerconst}^2,\paren*{\paren*{\matrowidx,\matcolidx} \in \ineqtwoboxset},}
        \end{mini!}
        which is obtained by replacing $\paren*{\J-\R}\Q$ with $\A$ in \cref{eq:1b2bRiemsysidentif}, respectively.
        Since this problem does not impose the stability condition on $\A$, the solutions are not necessarily stable.
The other is the Riemannian modeling that minimizes \eqref{eq:Cost1b2bRiemsysidentif} under the absence of any constraints.

        In the experiments, we consider the case $\dime=10, \sampintvl=0.02,$ and $\obsnum=40$. 
Indices of $\ineqoneboxset$ are randomly picked up from $\{1,2,3,\ldots, 10\}^2$, and then those of $\ineqtwoboxset$ are from $\{1,2,3,\ldots, 10\}^2\setminus \ineqoneboxset$. 
        We randomly generate $\paren*{\Jtrue,\Rtrue,\Qtrue}\in\mani$ using Manopt command M.rand() and then let the true system $\Atrue \coloneqq \paren*{\Jtrue - \Rtrue} \Qtrue \in \setR[\dime\times\dime]$.
        For each $(i,j)\in \mathcal{I}_1$, 
        $l_{ij}$ and $r_{ij}$ are randomly generated so that  $e_i^{\top}\Atrue e_j\in \sbra{\leftconst,\rightconst}$. 
        As well, for each $(i,j)\in\mathcal{I}_2$, 
        $l_{ij}, r_{ij}, k_{ij}$ and $c_{ij}$ are randomly generated so that $e_i^{\top}\Atrue e_j\in \sbra{\leftconst,\centerconst-\lengthconst}\cup\sbra{\centerconst+\lengthconst,\rightconst}$.
        We set the ratios of 1-box and 2-box constraints as $0.2$ and $0.1$, respectively.
        Each component of $\state[0]$ is generated uniformly at random in the range $(-1000,1000)$, 
         and $\brc{\state}_{\iteridx=1}^{\obsnum}$ 
        is determined according to $x_{k+1} = \exp (Ah)x_k$.
        Additive white Gaussian noise with a Signal-to-Noise Ratio (SNR) of $10$ dB or $20$ dB is then added to these values.
        We also scale the data by dividing all the elements by $\norm{\state[0]}$ for the sake of the numerical stability.
        We adopt randomly generated $\paren*{\J,\R,\Q}\in\mani$ as the initial point.
    \subsection{Experimental environment}\label{subsec:exptenv}
        We compare the following six pairs of modelings and algorithms:
        \begin{itemize}
            \item \textbf{\RNLObyRSQO{}}: our \RNLO{} modeling \cref{eq:1b2bRiemsysidentif} and \RSQO{} method proposed in \cite{Obaraetal20RiemSQO}.
            \item \textbf{EE}: the Euclidean nonlinear modeling \cref{eq:Euclisysidentif} and \fminconSQO{}, an Euclidean SQO or SQP solver.
            \item \textbf{EEP}: the Euclidean nonlinear modeling \cref{eq:Euclisysidentif} and \fminconSQO{} with the projection to the stable matrix set proposed in \cite{noferini2021nearest}.
            \item \textbf{\LS}: the Euclidean nonlinear modeling \cref{eq:CostEuclisysidentif} without the constraints and \fminconSQO{}.%
            \item \textbf{\GD}: the Riemannian modeling \eqref{eq:Cost1b2bRiemsysidentif} without the constraints and the Riemannian gradient decent method \cite{Absiletal08OptBook}.
            \item \textbf{\RTR}: the Riemannian modeling \eqref{eq:Cost1b2bRiemsysidentif} without the constraints and the Riemannian trust region method \cite{Absiletal08OptBook}.
        \end{itemize}
        
        We evaluate the quality of solutions of the optimization problems in view of the first largest real part of eigenvalues of $\A$, because it will govern the behavior of the system.
        We also monitor the optimization process of each algorithm with respect to the max violation; let
        \begin{align*}
            &\ineqleftfun\paren*{\prodvar} \coloneqq - \trsp{\stdbasis[\matrowidx]} \paren*{\J-\R}\Q \stdbasis[\matcolidx] + \leftconst, \quad \paren*{\paren*{\matrowidx,\matcolidx} \in \ineqoneboxset\cup\ineqtwoboxset},\\
            &\ineqrightfun\paren*{\prodvar} \coloneqq \trsp{\stdbasis[\matrowidx]} \paren*{\J-\R}\Q \stdbasis[\matcolidx] - \rightconst, \quad \paren*{\paren*{\matrowidx,\matcolidx} \in \ineqoneboxset\cup\ineqtwoboxset},\\
            &\ineqtwofun\paren*{\prodvar} \coloneqq - \paren*{\trsp{\stdbasis[\matrowidx]} \paren*{\J-\R}\Q \stdbasis[\matcolidx] - \centerconst}^2 + \lengthconst^2, \quad \paren*{\paren*{\matrowidx,\matcolidx} \in \ineqtwoboxset}%
        \end{align*}
        be the constraints in \cref{eq:IneqOne1b2bRiemsysidentif,eq:IneqTwo1b2bRiemsysidentif}.
        The max violation at $\prodvar$ is defined as
        \begin{align*}
            \max\brc*{0, \max_{\paren*{\matrowidx,\matcolidx} \in \ineqtwoboxset}\brc*{\ineqtwofun\paren*{\prodvar}}, \max_{\paren*{\matrowidx,\matcolidx} \in \ineqoneboxset\cup\ineqtwoboxset}\brc*{\ineqleftfun\paren*{\prodvar}, \ineqrightfun\paren*{\prodvar}}}.
        \end{align*}
        Throughout the experiments, we employ the identity mapping as $\oprrsqosub$ in \cref{def:objfunRSQOsubprob}, which clearly satisfies \cref{assu:opruniformbound}.

    \subsection{Numerical results and discussion}\label{subsec:exptresult}

We solved 50 problem instances by the six pairs of modelings and algorithms. 
All the algorithms were run from the identical initial point, which was produced for each problem instance, and terminated after 100 seconds of CPU time.
The results are depicted in Figs. \ref{fig:eigen_snr10}, \ref{fig:eigen_snr20}, \ref{fig:relerr}, and \ref{fig:maxvio}.
Note that the scales in the subfigures within each figure significantly differ, and the result shown in Fig. 1 was obtained by randomly selecting a solution from one of the 50 problem instances.
Although SNR was $10$ dB for Fig. \ref{fig:eigen_snr10}, it was $20$ dB for the other figures.
Figs. \ref{fig:eigen_snr10} and \ref{fig:eigen_snr20} demonstrate the eigenvalue distributions obtained by the various methods.
Fig. \ref{fig:relerr} denotes the boxplots of the relative errors of the first largest real part of eigenvalue between true system and numerical solutions among 50 trials, respectively.
Fig. \ref{fig:maxvio} denotes the boxplots of the max violation among 50 trials, respectively.

As shown in Figs. \ref{fig:eigen_snr10}, \ref{fig:eigen_snr20}, and \ref{fig:relerr}, \RNLObyRSQO{} was considerably better than the other methods
in terms of the identification of the eigenvalues.
In particular, Fig. \ref{fig:relerr} showed that only \RNLObyRSQO{} can approximately estimate the most important eigenvalue, that is, the eigenvalue with the first largest real part.
Here, we omitted the relative error attained by EE, because the error was considerably larger than those attained by the other methods.

While we haven't shown the figure due to page limitations, an important observation needs to be highlighted. Using the same initial point, EE's objective value was superior to that of RNLO-by-RSQO. However, in terms of eigenvalues, RNLO-by-RSQO considerably outperformed EE. Furthermore, when EE was initialized with the matrix $A$ obtained from RNLO-by-RSQO, its resulting eigenvalues were inferior to those generated by RNLO-by-RSQO. These findings emphasize EE's tendency to overfit noisy data, contrasting with the robustness demonstrated by RNLO-by-RSQO.

According to Fig. \ref{fig:maxvio}, 
\RNLObyRSQO{} was significantly better than the other methods except for EE in terms of the max violation.
This result can be expected, because only RNLO-by-RSQO and EE incorporate prior information in the form of equality and inequality constraints on the parameters.
The abundance of outliers implies that EE is susceptible to noise.
Note that the max violation of EEP, which projects the unstable eigenvalues obtained by EE onto the imaginary axis, is larger than that of EE.
This is because the projection does not take into account equality or inequality constraints.

 It is remarkable that \RNLObyRSQO{} is superior to \GD{} and \RTR{} in terms of the eigenvalues, as shown in Figs. \ref{fig:eigen_snr10}, \ref{fig:eigen_snr20}, and \ref{fig:relerr}.
 This results from the fact that \GD{} and \RTR{} consider to solve \eqref{eq:Cost1b2bRiemsysidentif} without the constraints unlike \RNLObyRSQO{}.
In other words, the problem for \GD{} and \RTR{} does not incorporate any prior knowledge other than stability unlike RNLO-by-RSQO.
 Thus, \GD{} and \RTR{} generated convergence sequences to some bad solutions as in the results of Figs. \ref{fig:eigen_snr10}, \ref{fig:eigen_snr20}, and \ref{fig:relerr}.

In summary, the superiority of RNLO-by-RSQO can be attributed to the parametrization based on Proposition \ref{prop:charactstab} and the incorporation of prior knowledge concerning equality and inequality constraints.
A particularly noteworthy point is that RNLO-by-RSQO is more robust to noise compared to EE.
The only difference between these methods is whether or not they utilize the parametrization based on Proposition \ref{prop:charactstab}.

\begin{figure}[htbp]
        \centering
        \includegraphics[keepaspectratio, scale=0.4] {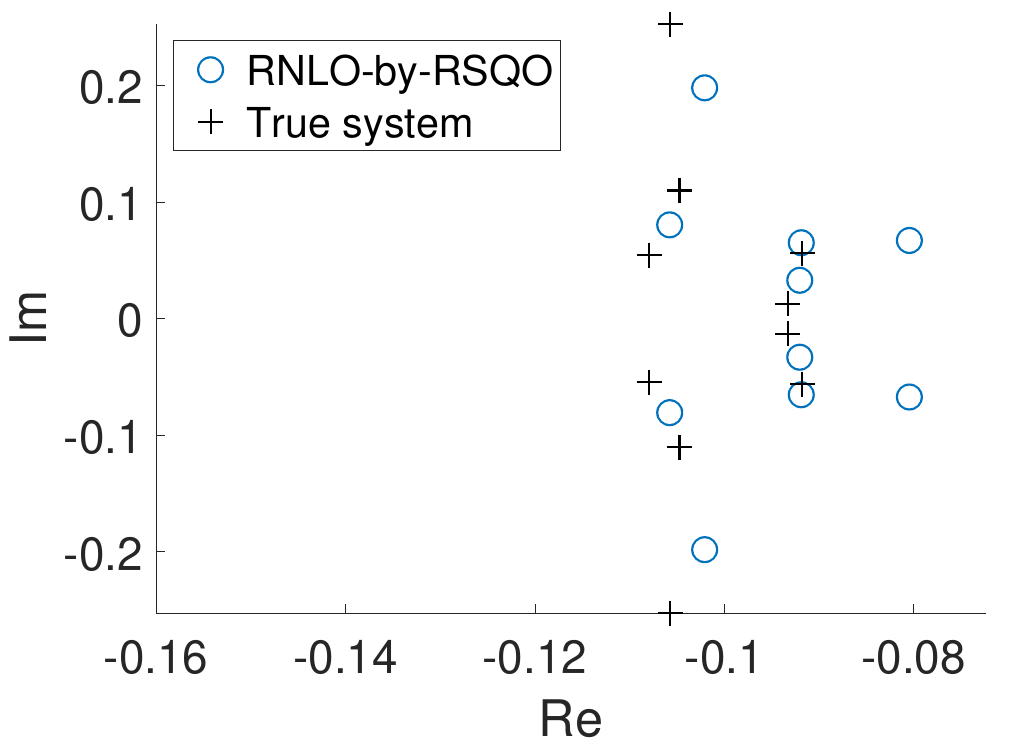}
     \vspace{0.5cm}
     
        \includegraphics[keepaspectratio, scale=0.4]{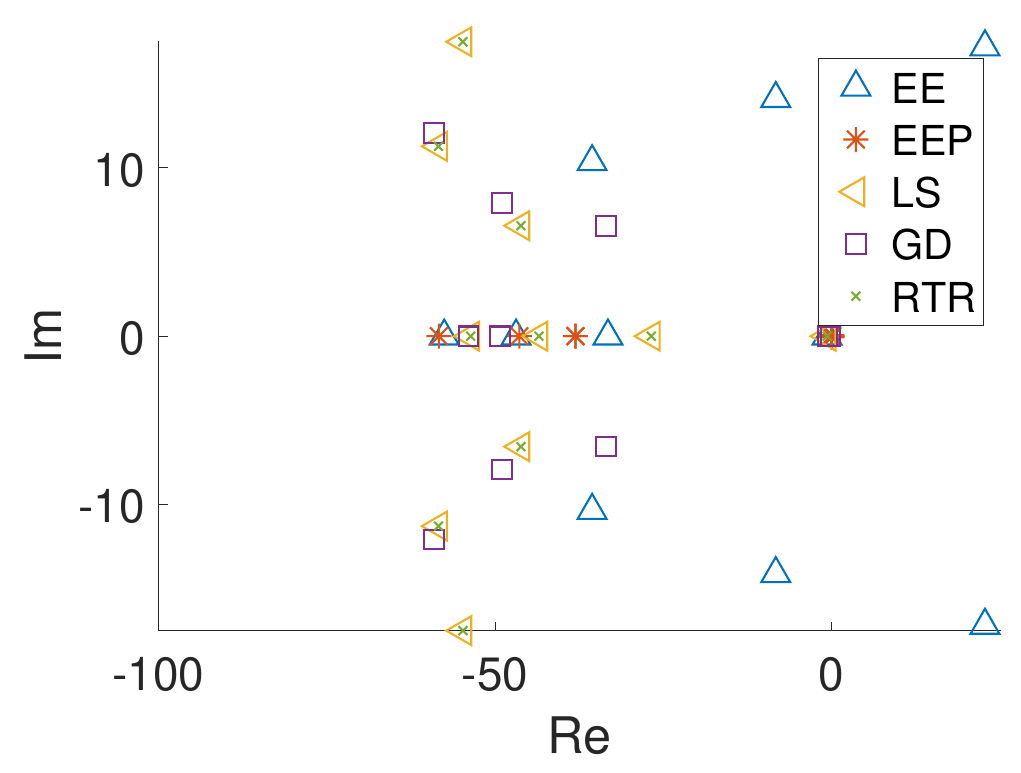}
    \caption{Eigenvalue distribution (SNR $=10$ dB)}
    \label{fig:eigen_snr10}
\end{figure}

\begin{figure}[htbp]
        \centering
        \includegraphics[keepaspectratio, scale=0.4]{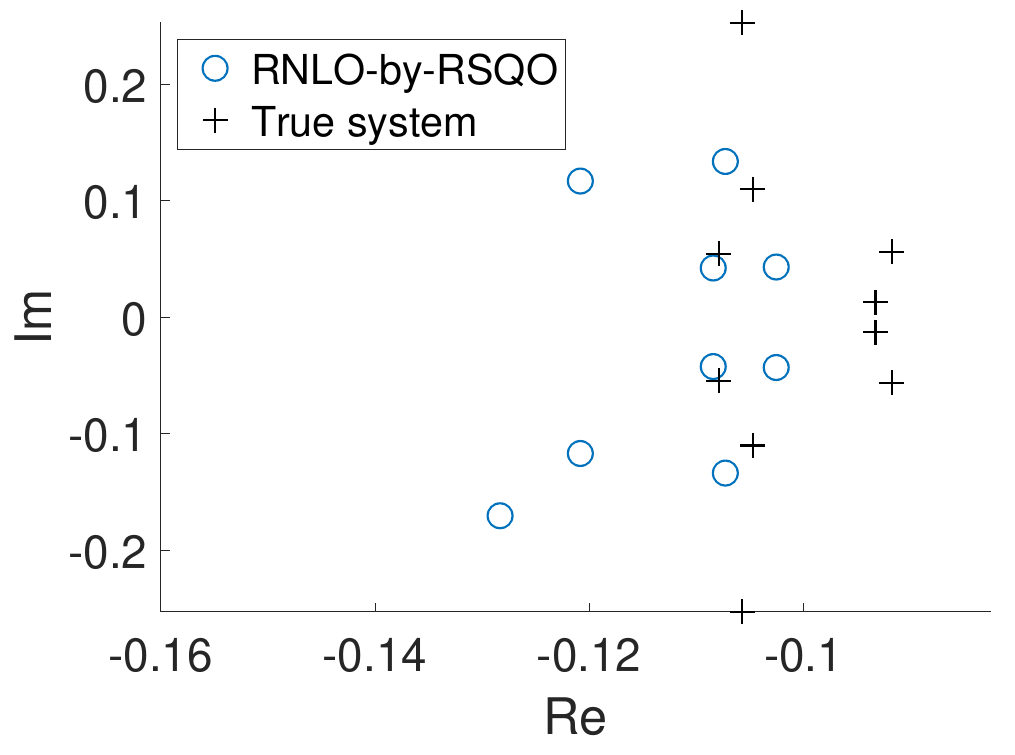}
      \vspace{0.5cm}
      
        \includegraphics[keepaspectratio, scale=0.4]{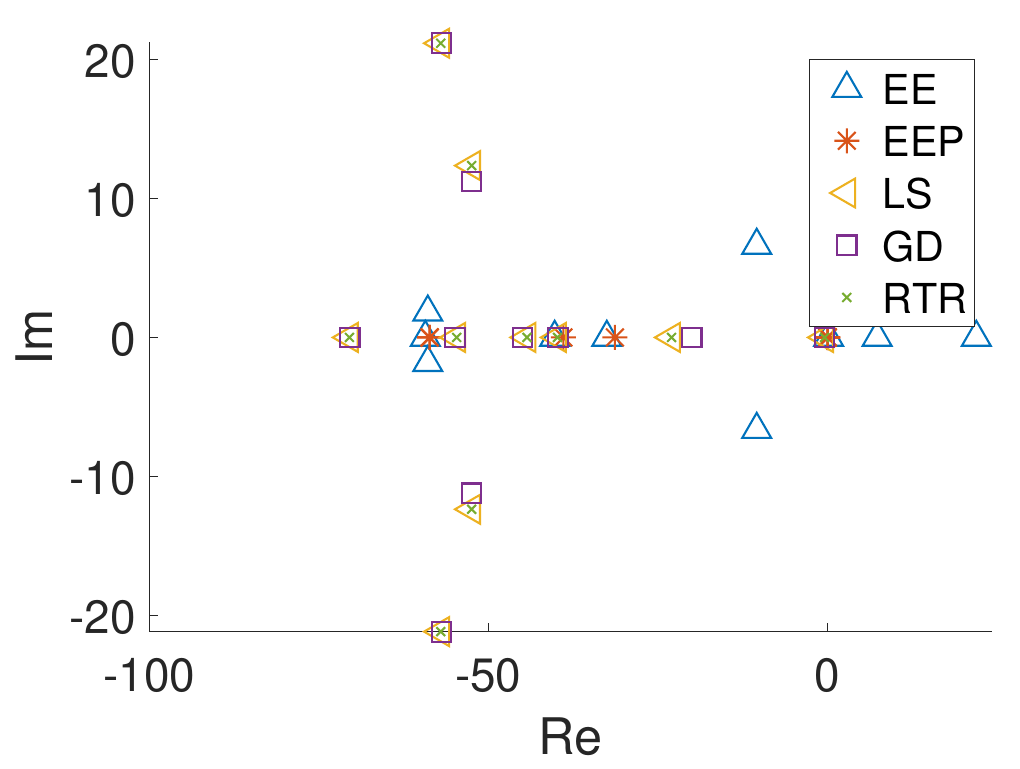}
    \caption{Eigenvalue distribution (SNR $= 20$ dB)}
    \label{fig:eigen_snr20}
\end{figure}

\begin{figure}[htbp]
        \centering
        \includegraphics[keepaspectratio, scale=0.34]{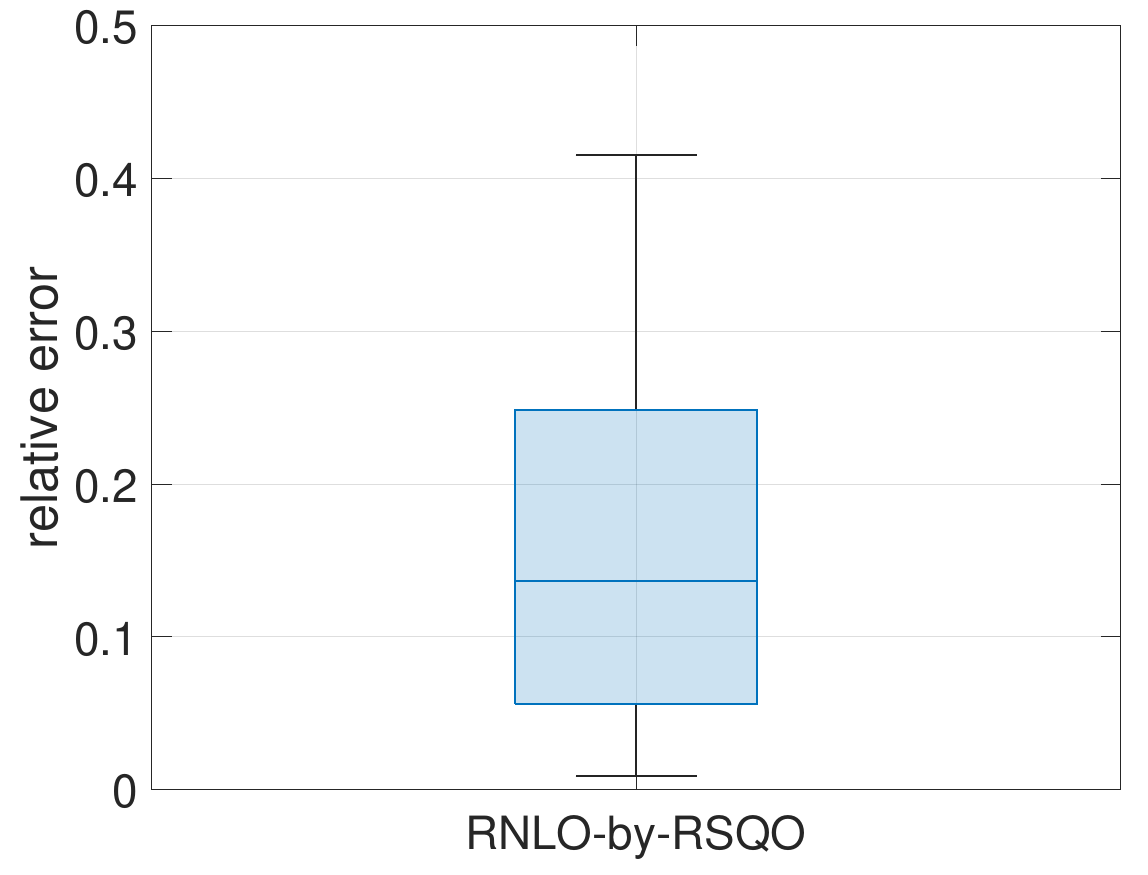}
\vspace{0.5cm}
   
        \includegraphics[keepaspectratio, scale=0.34]{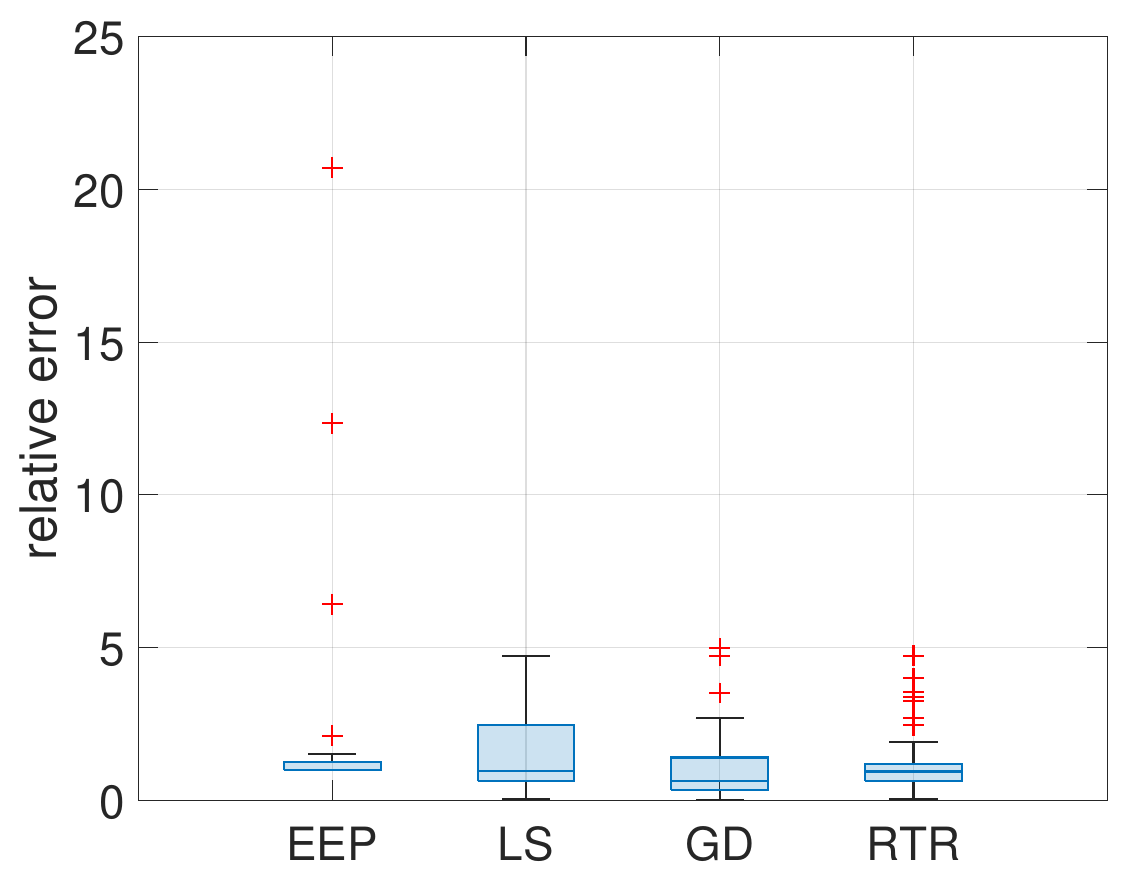}
    \caption{Relative error of the first largest real part of the eigenvalue among 50 trials (SNR $= 20$ dB)}
    \label{fig:relerr}
\end{figure}

\begin{figure}[htbp]
        \centering
        \includegraphics[keepaspectratio, scale=0.34]{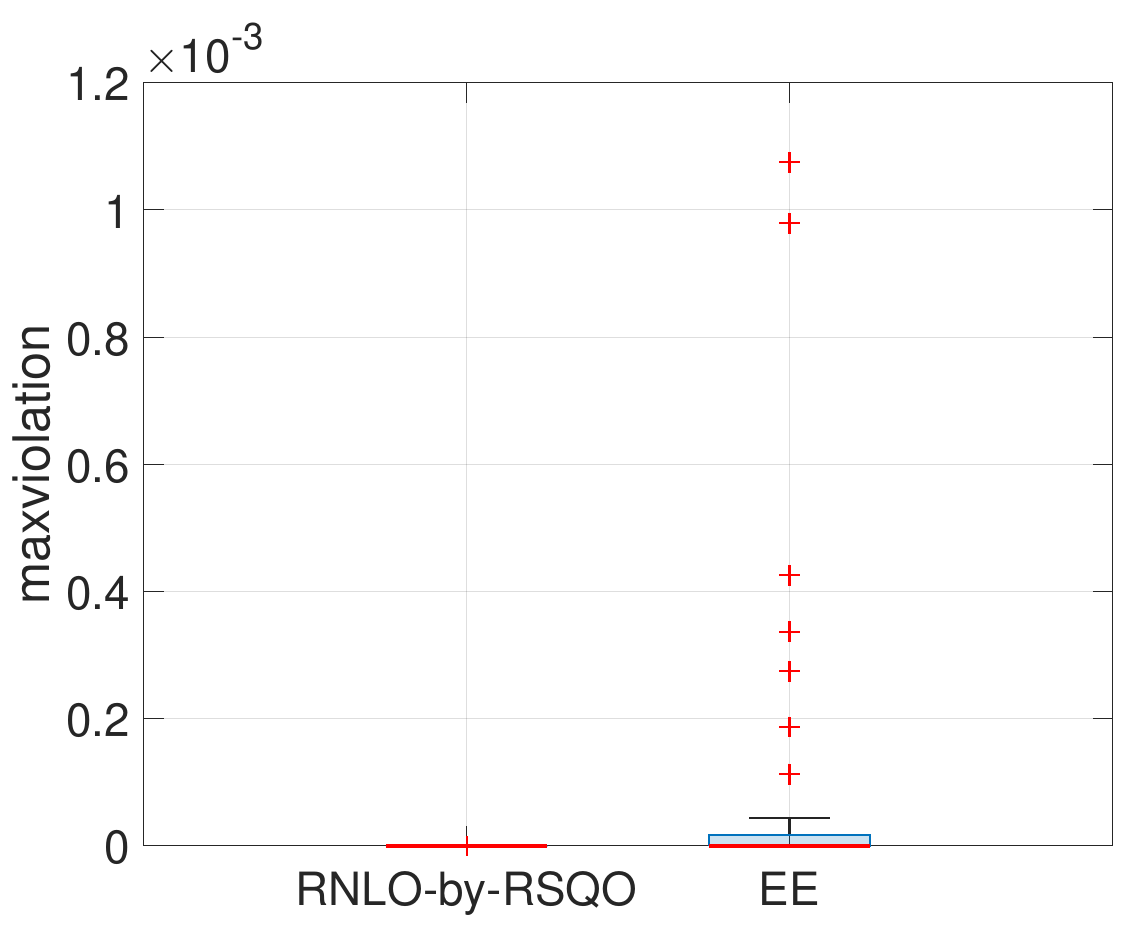}
   
      \vspace{0.5cm}
      
        \includegraphics[keepaspectratio, scale=0.34]{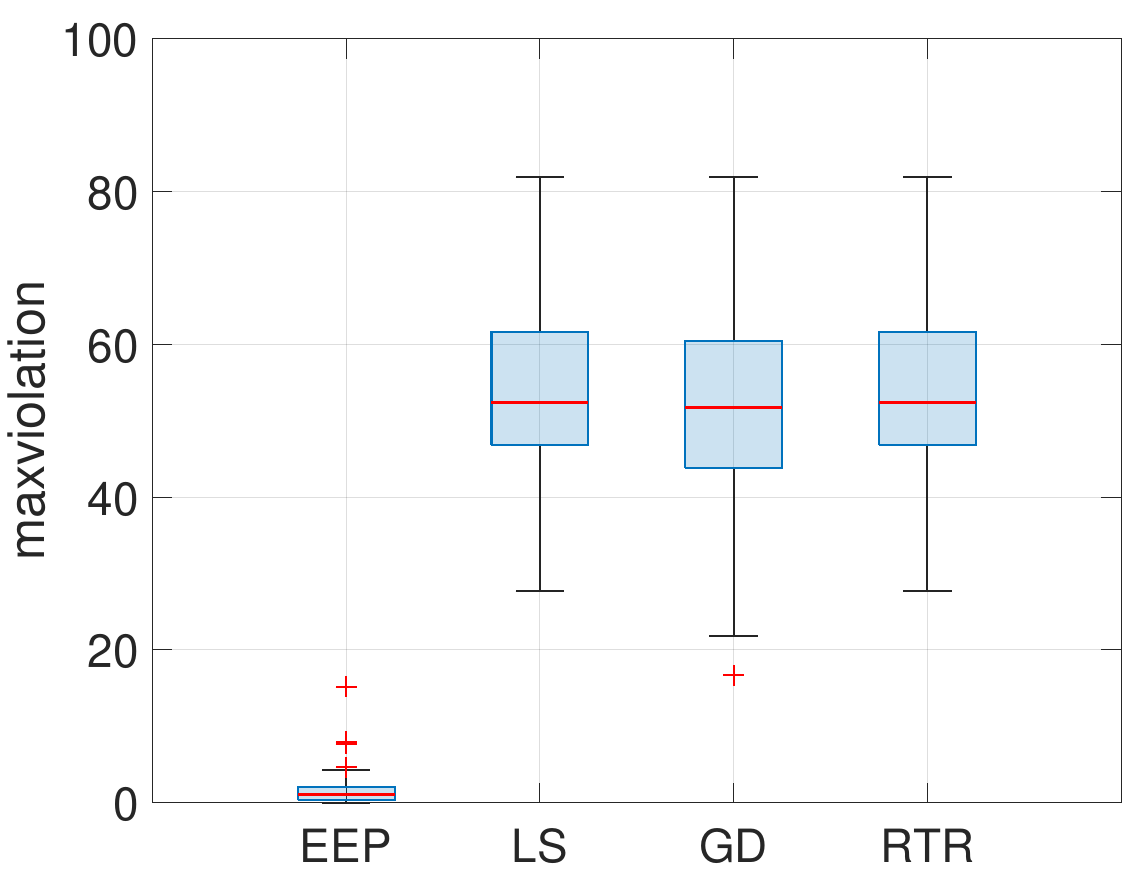}
    \caption{The max violation among 50 trials (SNR $= 20$ dB)}
       \label{fig:maxvio}
\end{figure}
    
    \section{Conclusion and future work}\label{sec:conclusion}
        We developed a prediction error method that ensures the stability of a linear system and meets with the prior knowledge.
        The method employs \RNLO{}, a class of the nonlinear optimization on a Riemannian manifold. For solving this \RNLO{},
        we proposed to use \RSQO{}.
        Numerical experiments with comparisons declare the effectiveness of our \RNLO{} formulation in terms of the stability of the system and prior knowledge together with the superiority of \RSQO{}.
        
        An initial point generally affects numerical results in nonlinear optimization, 
        and a sampling time $h$ also does the accuracy of the system identification. Hence, a good proposal for the selection of an initial point and a sampling time is a possible future research.
        Moreover, it should be noted that our approach presented in this paper does not always identify a stable system with prior knowledge information and sufficient accuracy.
        In fact, for example, in situations where there is little data and a large amount of noise, the estimated matrix using our approach can be a stable matrix near the imaginary axis on the complex plane even if the true eigenvalues are far away from the axis.
        This issue should be resolved in the future research.
        
        \printbibliography
\end{document}